\documentclass[12pt]{article}
\usepackage{amssymb,amsfonts, amsmath,amsthm,epsfig,tabularx}
\usepackage{caption}
\usepackage{float}
\restylefloat{table}

\usepackage{hyperref}

\usepackage{caption} 
\usepackage{subcaption} 

\newtheorem{theorem}{Theorem}
\newtheorem{proposition}[theorem]{Proposition}
\newtheorem{lemma}[theorem]{Lemma}

\newtheorem{observation}[theorem]{Observation}

\newtheorem{problem}[theorem]{Problem}

\newcommand{\ivs}{{\rm ivs_{\chi}}}
\newcommand{\vs}{{\rm vs_{\chi}}}
\newcommand{\C}{{C_{\chi}}}
\newcommand{\Aut}{{\rm Aut}}

\begin{document}

\title{On chromatic vertex stability of 3-chromatic graphs with maximum degree 4}
\author{Martin Knor$^{1}$, Mirko Petru\v{s}evski$^{2}$,
Riste \v{S}krekovski$^{3,4}$ \\[0.3cm] {\small $^{1}$ \textit{Slovak University of Technology in Bratislava,
Bratislava, Slovakia}}\\[0.1cm] {\small $^{2}$ \textit{University Ss Cyril and Methodius, MFS, 1000 Skopje, Macedonia }}\\[0.1cm] {\small $^{3}$ \textit{University of Ljubljana, FMF, 1000 Ljubljana,
Slovenia }}\\[0.1cm] {\small $^{4}$ \textit{Faculty of Information Studies, 8000 Novo
Mesto, Slovenia }}\\[0.1cm] }

\maketitle

\begin{abstract}
The (independent) chromatic vertex stability ($\ivs(G)$) $\vs(G)$ is the minimum size of (independent) set $S\subseteq V(G)$ such that $\chi(G-S)=\chi(G)-1$.
In this paper we construct infinitely many graphs $G$ with $\Delta(G)=4$, $\chi(G)=3$, $\ivs(G)=3$ and $\vs(G)=2$, which gives a partial negative answer to a problem posed in \cite{ABKM}.
\end{abstract}

%
%
\section{Introduction}

Let $G$ be a graph.
Its edge stability number, $\rm{es}_{\chi}(G)$, is the minimum number of edges whose deletion results in a graph $H$ with $\chi(H)=\chi(G)-1$.
Edge stability number was introduced in 1980 by Staton~\cite{S}, and rediscoverd in 2008 by Arumugam, Sahul Hamid and Muthukamatchi~\cite{ASM}.
For recent results on this invariant see e.g. \cite{ABBDMM, AKMN, BKM, KMM}.

General concept of stability number appeared in~\cite{BHNS}, but the first paper on chromatic vertex stability number was written by Akbari, Beikmohammadi, Klav{\v z}ar and Movarraei in 2021, see~\cite{ABKM}.
The chromatic vertex stability $\vs(G)$ of $G$ is the minimum number of vertices of $G$ such that their deletion results in a graph $H$ with $\chi(H)=\chi(G)-1$.
Analogously, the independent chromatic vertex stability $\ivs(G)$ of $G$ is the minimum number of independent vertices  of $G$ such that their deletion results in a graph $H$ with $\chi(H)=\chi(G)-1$. Obviously, $\vs(G)\leq\ivs(G)$.
The main result of \cite{ABKM} is the following.

\begin{theorem}
\label{thm:ABKM}
If $G$ is a graph with $\chi(G)\in\{\Delta(G),\Delta(G)+1\}$ then $\vs(G)=\ivs(G)$.
\end{theorem}

The authors showed that as soon as $\chi(G)\le\frac{\Delta(G)+1}2$ the equality $\vs(G)=\ivs(G)$ need no longer be true. Whether $\vs(G)=\ivs(G)$ always holds if $\frac{\Delta(G)}{2}+1\leq\chi(G)\leq\Delta(G)-1$ was left unanswered, so they asked the following question (see Problem~3.2 in~\cite{ABKM}).

\begin{problem}
\label{prob:main}
Is it true that if $G$ is a graph with $\chi(G)\ge\frac{\Delta(G)}2+1$, then $\vs(G)=\ivs(G)$?
\end{problem}

In this paper we prove that Problem~\ref{prob:main} answers in the negative for $\chi(G)=3$ and $\Delta(G)=4$. A simple `ladder-like' counterexample on $9$ vertices is depicted in Figure~\ref{ladder}. Notice that the ladder part of the counterexample can be of any length $4k$ for $k\ge2$.

\begin{figure}[htp!]
	$$
    \includegraphics[scale=0.55]{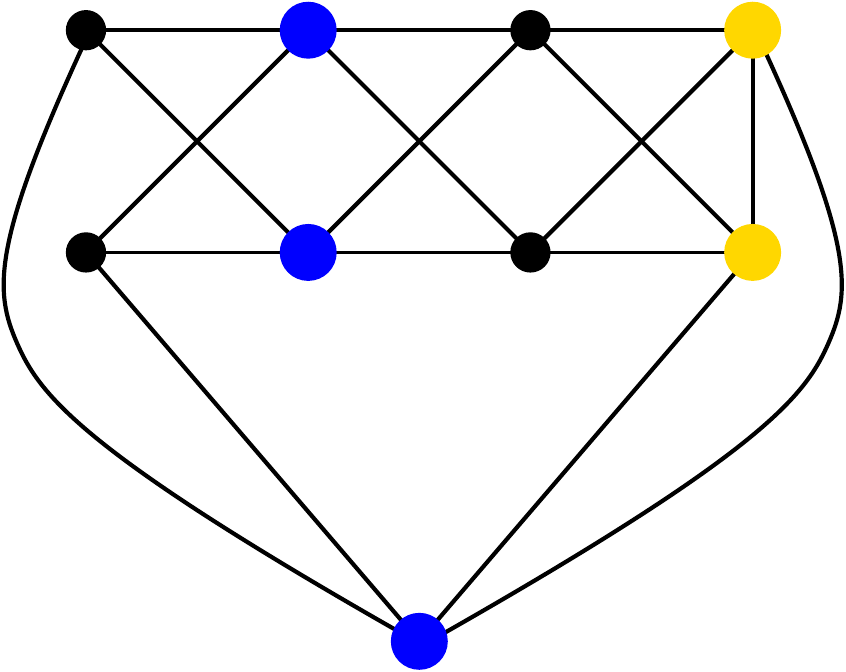}
	$$
	\caption{A graph $G$ with $\Delta(G)=4$ and $\chi(G)=3$. The blue independent $3$-set realizes $\ivs(G)=3$ and the yellow $2$-set realizes $\vs(G)=2$.}
	\label{ladder}
\end{figure}

Our main result is the following.

\begin{theorem}
\label{thm:main}
For each $n\ge 9$ there are at least $\max\{1,2^{\lfloor\frac{n-11}2\rfloor}\}$ planar graphs with $\chi(G)=3$, $\Delta(G)=4$, $\ivs(G)=3$ and $\vs(G)=2$.
\end{theorem}

Observe that if $\chi(G)=3$ and $\Delta(G)=4$, then $\chi(G)=\frac{\Delta(G)}2+1=\Delta(G)-1$.
Hence, the bound on $\chi(G)$ in Theorem~{\ref{thm:ABKM}} cannot be relaxed
when $\chi(G)=3$.

%
%
\section{Proofs}

We start with a pair of simple observations followed by a couple of lemmas.

\begin{observation}
    \label{obs:1}
For every graph $G$, $\ivs(G)$ equals the minimum size of a colour class over all proper $\chi(G)$-colourings of $G$. Hence $|V(G)|\geq\ivs(G)\cdot\chi(G)$.
\end{observation}
\begin{proof}
Let us first notice that there exists a proper $\chi(G)$-colouring of $G$ with a colour class of size $\ivs(G)$. Indeed, take $S\subseteq V(G)$ to be an independent set of size $|S|=\ivs(G)$ such that $\chi(G-S)=\chi(G)-1$.
Use a proper $(\chi(G)-1)$-colouring of $G-S$ and assign to vertices of $S$ a new colour.
So $\ivs(G)$ is not less than the minimum size of a colour class over all proper $\chi(G)$-colourings of $G$.

Contrarily, consider a proper $\chi(G)$-colouring of $G$ which minimizes the size of a colour class, and let $S$ be such a minimum colour class. Then $S$ is an independent subset of $V(G)$ and $\chi(G-S)\leq\chi(G)-1$. In fact, we must have equality here for otherwise $G$ would admit a proper $(\chi(G)-1)$-colouring. So $\ivs(G)$ is also not more than the minimum size of a colour class over all proper $\chi(G)$-colourings of $G$, which proves our point.

The inequality $|V(G)|\geq\ivs(G)\cdot\chi(G)$ is now an immediate consequence.
\end{proof}

\begin{observation}
    \label{obs:2}
If $\Delta(G)\leq2$ then $\vs(G)=\ivs(G)$.
\end{observation}
\begin{proof}
We may assume that $\vs(G)\geq2$.
Indeed, if $\vs(G)=1$ then obviously $\ivs(G)=1$ as well.
We may also assume that $G$ is connected.
Then $G$ is either a path or an even cycle.
In either case $\vs(G)=\ivs(G)=\lfloor \frac{|V(G)|}{2}\rfloor$.
\end{proof}

As already mentioned, we are interested in finding graphs $G$ for which $\chi(G)\ge\frac{\Delta(G)}2+1$ and $\ivs(G)>\vs(G)$. Our first lemma establishes some implications for the order and the considered stability parameters.

\begin{lemma}
\label{lem:9}
If $\ivs(G)>\vs(G)$ and $\chi(G)\ge\frac{\Delta(G)}2+1$ then $|V(G)|\ge 9$, $\ivs(G)\ge 3$, $\vs(G)\ge 2$ and $\chi(G)\ge3$. Moreover, if $|V(G)|=9$ then $\ivs(G)=3$, $\vs(G)=2$ and $\chi(G)=3$.
\end{lemma}

\begin{proof}
Since $\ivs(G)>\vs(G)$, we must have $\vs(G)\ge 2$ and consequently $\ivs(G)\ge 3$.

If $\chi(G)\le 2$ then from $\chi(G)\ge\frac{\Delta(G)}2+1$ we get $\Delta(G)\le 2(\chi(G)-1)\le 2$, which in view of Observation~\ref{obs:2} contradicts $\ivs(G)>\vs(G)$.
Hence $\chi(G)\ge 3$.

From the inequality stated in Observation~\ref{obs:1}, it follows that
$|V(G)|\ge\ivs(G)\cdot\chi(G)\ge 3\cdot 3$, that is, $|V(G)|\ge 9$.
And if $|V(G)|=9$ then $\chi(G)=\ivs(G)=3$ and $\vs(G)=2$.
\end{proof}

Figure~\ref{fig:g9g10} depicts two graphs, respectively denoted by $G_9$ and $G_{10}$ in regard to their orders. The former one can be obtained from the 1-skeleton of a regular octagon by subdividing the edges of a triangle. It has $\Delta(G_9)=4$, $\chi(G_9)=\ivs(G_9)=3$ and $\vs(G_9)=2$. Observe that $\chi(G_9-\{x,y\})=2$ if and only if $\{x,y\}=\{v_i,v_j\}$, where $1\le i<j\le 3$, and for an independent set of vertices $\{x,y,z\}$ we have $\chi(G_9-\{x,y,z\})=2$ if and only if $\{x,y,z\}=\{u_i,v_i,w_i\}$, where $1\le i\le 3$. The graph $G_{10}$ is obtained from $G_9$ by adding the vertex $q$ and connecting it to $w_2$ and $w_3$.
It also has $\Delta(G_{10})=4$, $\chi(G_{10})=\ivs(G_{10})=3$ and $\vs(G_{10})=2$. Again $\chi(G_{10}-\{x,y\})=2$ if and only if $\{x,y\}=\{v_i,v_j\}$, where $1\le i<j\le 3$, and for an independent set of vertices $\{x,y,z\}$ we have $\chi(G_{10}-\{x,y,z\})=2$ if and only if $\{x,y,z\}=\{u_i,v_i,w_i\}$ where $2\le i\le 3$.

\begin{figure}[htp!]
	$$
	\includegraphics[scale=0.25]{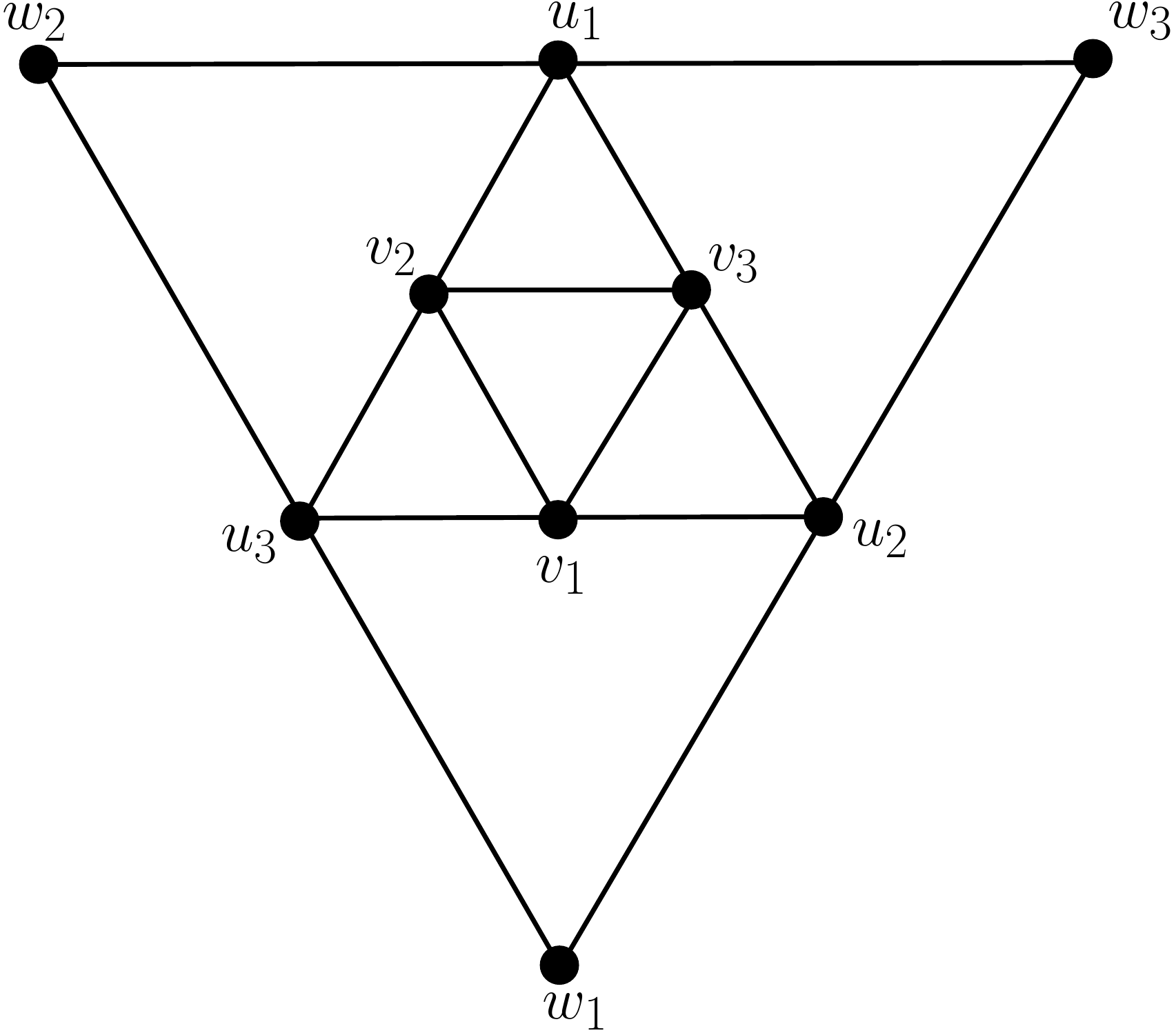} \quad\quad\quad\quad
    \includegraphics[scale=0.3]{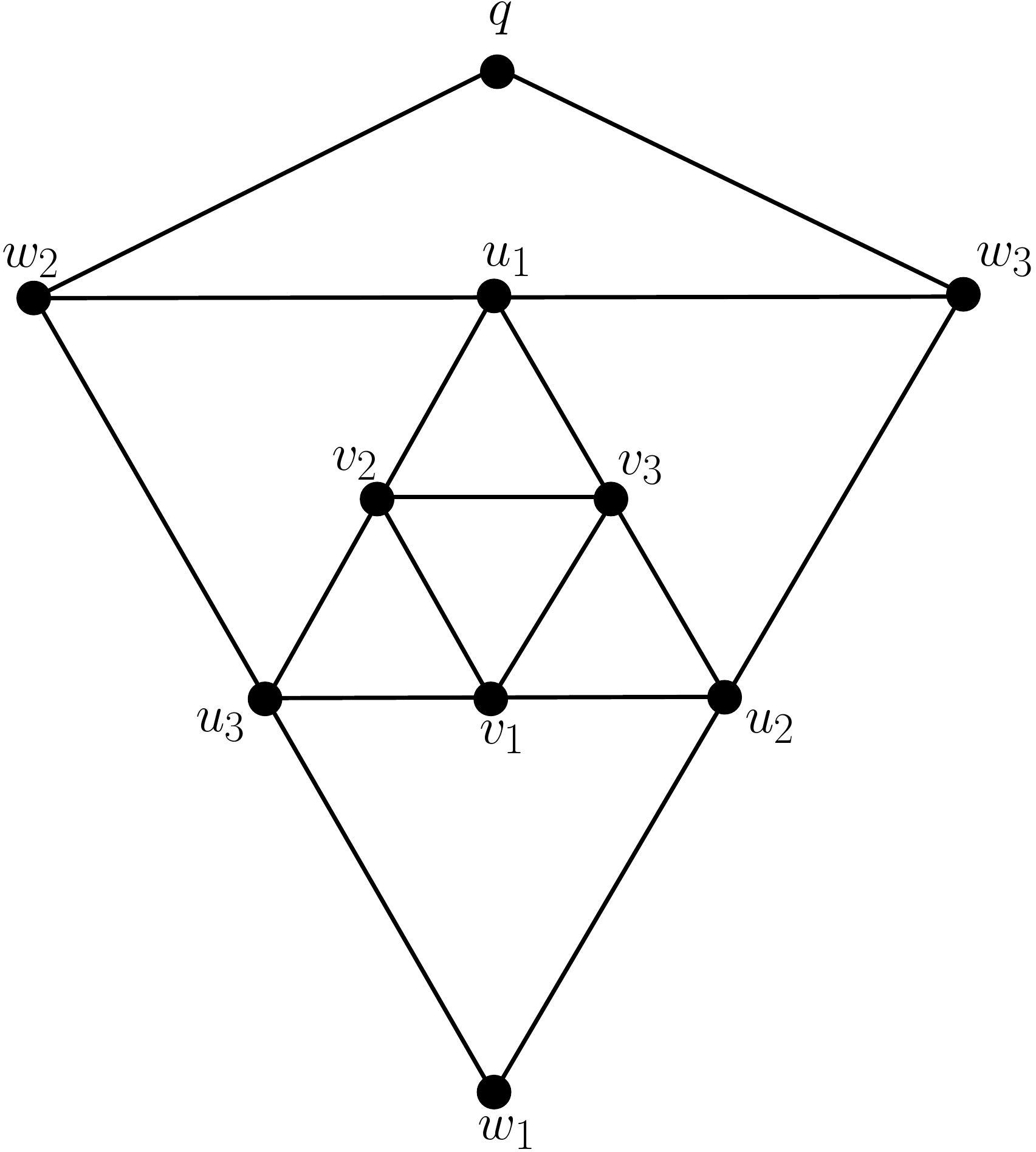}
	$$
	\caption{The graph $G_9$ (left) and the graph $G_{10}$ (right).}
	\label{fig:g9g10}
\end{figure}

Our second lemma concerns the case $\chi(G)=\ivs(G)=3$ and $\vs(G)=2$. Under the assumption $\chi(G)\ge\frac{\Delta(G)}2+1$ we establish the maximum degree of $G$ and the vertex degrees of every $2$-set which realizes $\vs(G)$.

\begin{lemma}
\label{lem:d=4}
Let $G$ be a graph with $\chi(G)=\ivs(G)=3$, $\vs(G)=2$ and $\chi(G)\ge\frac{\Delta(G)}2+1$.
Then $\Delta(G)=4$ and for every $v_1,v_2\in V(G)$ such that $\chi(G-\{v_1,v_2\})=2$ we have $\deg_G(v_1)=\deg_G(v_2)=4$.
\end{lemma}

\begin{proof}
Let $G$ satisfy the assumptions of the lemma and let $v_1,v_2\in V(G)$ such that $\chi(G-\{v_1,v_2\})=2$. Then $\Delta(G)\le2(\chi(G)-1)=4$.
By way of contradiction, suppose that $\deg_G(v_1)\le 3$.
Since $\ivs(G)=3>2$, we have $v_1v_2\in E(G)$.
Further, since $\chi(G-v_2)=3$, there must be an odd cycle in $G-v_2$; moreover,
every such cycle passes through $v_1$ as $G-\{v_1,v_2\}$ is bipartite.
Consequently $\deg_G(v_1)=3$.

Let $u_1,u_2$ be neighbours of $v_1$ in $G-v_2$. Since every odd cycle in $G-v_2$ passes through both $u_1,u_2$, we conclude that
$G-\{u_i,v_2\}$ is bipartite as well, $i\in\{1,2\}$. From $\ivs(G)=3>2$ it follows that $u_1v_2,u_2v_2\in E(G)$.
Moreover, $u_1u_2\notin E(G)$, for otherwise $v_1,v_2,u_1,u_2$ induces $K_4$, implying $\chi(G)\ge 4$.

Since $\ivs(G)\ge 3$, there must be an odd cycle in $G-\{u_1,u_2\}$.
This cycle cannot pass through $v_1$ since $\deg_{G-\{u_1,u_2\}}(v_1)=1$.
If this cycle does not pass through $v_2$ as well then it is in $G-\{v_1,v_2\}$ which means that $\chi(G-\{v_1,v_2\})\ge 3$, a contradiction.
Hence, there is an odd cycle passing through $v_2$ in $G-\{u_1,u_2,v_1\}$, which means that $\deg_G(v_2)\ge 5$. This contradiction settles the lemma.
\end{proof}

 A computer search shows that there are precisely $30$ graphs $G$ of order $9$ and having $\Delta(G)=4$, $\chi(G)=3$,  $\ivs(G)=3$ and $\vs(G)=2$.
Several of them (including $G_9$) are planar and four are obtained by adding an edge to another graph from the same collection.

\medskip

For every $n\ge9$ let $S_n$ be the set of graphs $G$ on $n$ vertices such that $\Delta(G)=4$, $\chi(G)=3$, $\ivs(G)=3$ and $\vs(G)=2$. Thus $G_9\in S_9$ and $G_{10}\in S_{10}$.
By $\C(G)$ we denote the set of vertices $x\in V(G)$ such that there is $y\in V(G)$ for which $\chi(G-\{x,y\})=2$; note in passing that every such $y$ is a neighbour of $x$.
For example, $\C(G_9)=\C(G_{10})=\{v_1,v_2,v_3\}$. In view of our next result, for every $n\ge9$ there is a planar graph in $S_n$ which is topologically equivalent to $G_9$ or $G_{10}$.

\begin{proposition}
\label{pro:subdivide}
Let $G\in S_n$ and
$e_1,e_2,\dots,e_t\in E(G)\backslash E([\C(G)])$, i.e., each $e_i$ has at most one endvertex in $\C(G)$. Let $n_1,n_2,\dots,n_t$ be positive even integers.
For every $i$, $1\le i\le t$, subdivide $e_i$ with $n_i$ new vertices, and denote the resulting graph by $H$.
Then $\Delta(H)=4$, $\chi(H)=3$, $\ivs(H)=3$ and $\vs(H)=2$. In other words, $H\in S_{n+(n_1+\cdots+n_t)}$.
\end{proposition}

\begin{proof}
Obviously $\Delta(H)=\Delta(G)=4$.
As $\chi(G)=3$, the graph $G$ has an odd cycle.
Since to any edge of this cycle we added an even number (possibly zero) of vertices, $H$ also has an odd cycle; thus $\chi(H)\ge 3$.
Moreover, if $S\subseteq V(G)$ is such that $G-S$ is bipartite then $H-S$ is bipartite as well.
Hence, $\chi(H)=3$, $\ivs(H)\le 3$ and $\vs(H)\le 2$.

If there is $v\in V(H)$ such that $\chi(H-v)=2$, then $v\notin V(G)$.
So $v$ is obtained by subdividing an edge, say $xy$, of $G$.
However, as $H-v$ is bipartite, both $G-x$ and $G-y$ are bipartite, a contradiction.
Hence, $\vs(H)=2$.

Finally, let us show that $\ivs(H)= 3$. Supposing the opposite, there are $u,v\in V(H)$ such that $\chi(H-\{u,v\})=2$ and $uv\notin E(H)$. It cannot be that both $u$ and $v$ are in $V(G)$, because we did not subdivide edges connecting vertices of
$\C(G)$.
So we may assume that $v$ is obtained by subdividing an edge $xy$ of $G$, where $y\notin\C(G)$.
Since $H-\{u,v\}$ is bipartite, so is $H-\{u,y\}$.
But then $u$ cannot be a vertex of $G$ as well.
Hence, $u$ is obtained by subdividing an edge $wz$ of $G$.
As $H-\{u,y\}$ is bipartite, so is $H-\{z,y\}$. However, this contradicts the fact that $y\notin\C(G)$.
\end{proof}

From Proposition~\ref{pro:subdivide} we deduce that $S_n\neq\emptyset$ for every $n\ge9$. Indeed, if $n$ is odd then take $G_9$, subdivide the edge $u_2w_1$ with $n-9$ new vertices and denote the resulting graph by $G_n$.
Analogously if $n$ is even then take $G_{10}$, subdivide the edge $u_2w_1$ with $n-10$ new vertices and denote the resulting graph by $G_n$.
Then $G_n$ is a connected planar graph and $G_n\in S_n$, by Proposition~{\ref{pro:subdivide}}.

\smallskip

Our next result shows that $S_n$ contains exponentially many planar graphs. Its proof relies on the following construction. We  pair up several new vertices of $H$, i.e., vertices belonging in $V(H)\backslash V(G)$ to obtain a graph $F$  such that both $F-\{v_1,v_2\}$ and $F-\{u_1,u_2,u_3\}$ are bipartite.
Then $F\in S_{n'}$ for some $n'$.

\begin{theorem}
\label{thm:many}
For each $n\ge 11$ there are at least $2^{\lfloor\frac{n-11}2\rfloor}$ 2-connected planar graphs in $S_n$.
\end{theorem}

\begin{proof}
In view of $G_{11}$ and $G_{12}$, we assume $n\ge 13$.
Take $G_n$ and relabel the vertices of the $u_2-u_3$ path that passes through $w_1$ by $u_2=a_0,a_1,a_2,\dots,a_{\ell-1}=w_1,a_{\ell}=u_3$; here $\ell=n-7$ if $n$ is odd and $\ell=n-8$ if $n$ is even.
Let $E_n=\{a_1a_{\ell-2},a_2a_{\ell-3},\dots,a_{\ell/2-2}a_{\ell/2+1}\}$. For every $E'\subseteq E_n$,
denote by $H_{n,E'}$ the graph obtained from $G_n$ by adding the edges of $E'$.
Obviously $H_{n,E'}$ is planar, $\Delta(H_{n,E'})=4$ and $\chi(H_{n,E'})=3$.
Moreover, $H_{n,E'}-\{x,y\}$ is bipartite if $\{x,y\}=\{v_i,v_j\}$ where $1\le i<j\le 3$, which implies that $\vs(H_{n,E'})=2$.
Also, $H_{n,E'}-\{u_2,v_2,w_2\}$ is bipartite which gives $\ivs(H_{n,E'})\le 3$.
On the other hand, since $G_n$ is a subgraph of $H_{n,E'}$ and $\ivs(G_n)=3$, we have $\ivs(H_{n,E'})=3$ as well.
Thus $H_{n,E'}\in S_n$.

\begin{figure}[ht!]
	$$
    \includegraphics[scale=0.25]{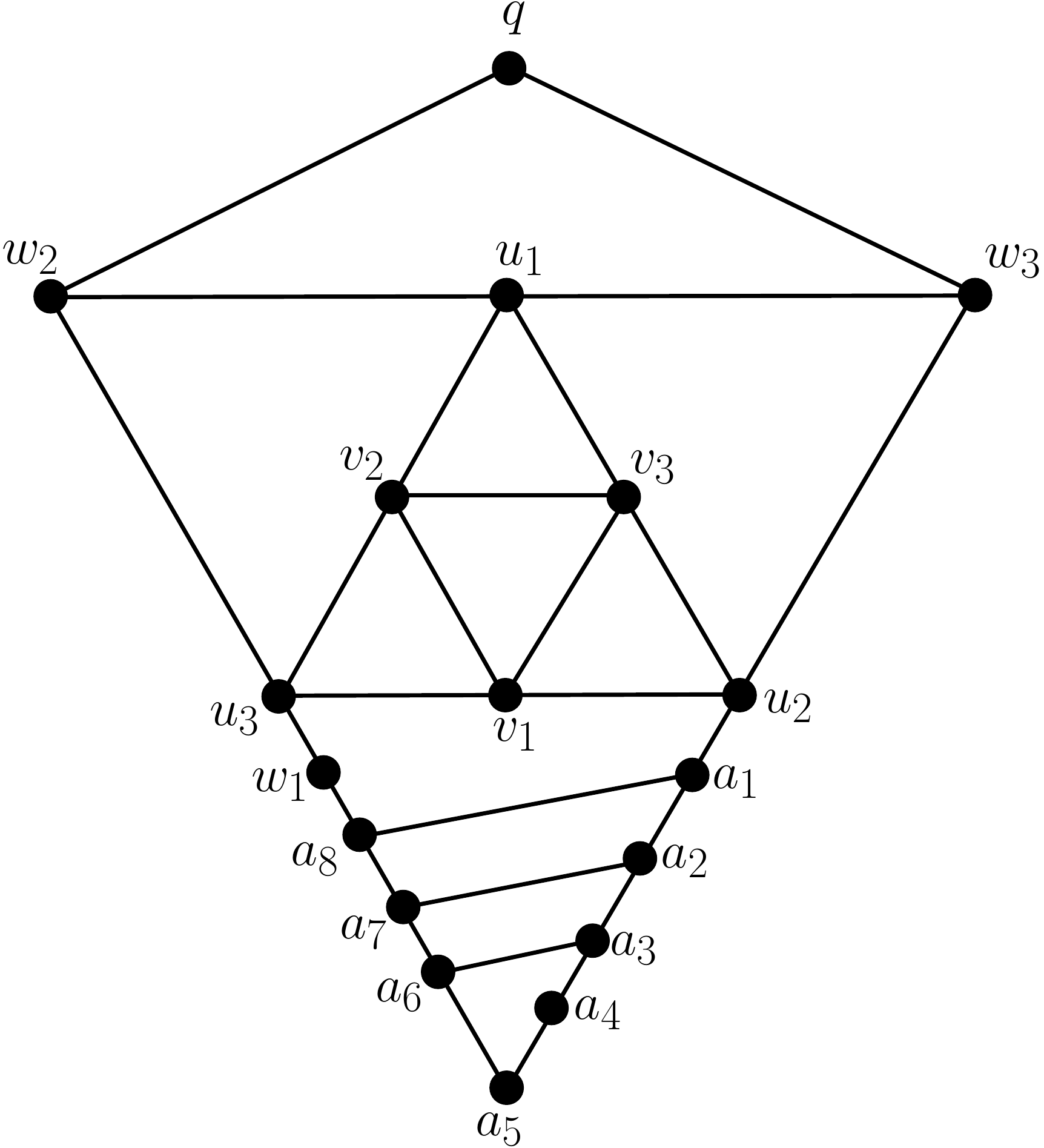}
	$$
	\caption{The graph $H_{18,E_{18}}$.}
	\label{fig:h18}
\end{figure}

Let $E',E^*\subseteq E_n$, where $E'\ne E^*$.
We show that the graphs $H_{n,E'}$ and $H_{n,E^*}$ are not isomorphic.
This is obvious if $|E'|\ne |E^*|$.
So assume that $|E'|=|E^*|\ge 1$.
We show that $\Aut(H_{n,E'})$, the group of automorphisms of $H_{n,E'}$, (and also $\Aut(H_{n,E^*})$) is trivial.
That is, every automorphism of $H_{n,E'}$ fixes all the vertices of $H_{n,E'}$.

There are exactly $6$ vertices of degree $4$ in $H_{n,E'}$, namely $u_1,u_2,u_3,v_1,v_2,v_3$.
Since each of $u_1,u_2,u_3$ is in only one triangle in $H_{n,E'}$ whereas each of $v_1,v_2,v_3$ is in three such triangles, every automorphism must preserve the sets $\{u_1,u_2,u_3\}$ and $\{v_1,v_2,v_3\}$.
The vertices $u_1$ and $u_2$ are both adjacent to a vertex of $H_{n,E'}-\{v_1,v_2,v_3\}$.
Also the vertices $u_1$ and $u_3$ are both adjacent to a vertex of $H_{n,E'}-\{v_1,v_2,v_3\}$.
But $u_2$ and $u_3$ are not adjacent to a vertex of $H_{n,E'}-\{v_1,v_2,v_3\}$, because $n\ge 13$. Consequently, every automorphism of $H_{n,E'}$ fixes $u_1$.

In view of  $a_{\ell-1}(=w_1)$, the vertex $u_3$ has a neighbour of degree $2$ which is not connected to $u_1$.
If $u_2$ does not have such a neighbour, then every automorphism of $H_{n,E'}$ fixes also $u_2$ and $u_3$.
So assume that also $u_2$ has a neighbour of degree $2$ which is not connected to $u_1$.
Now start at $u_2$, proceed with the above mentioned neighbour of $u_2$ and construct a longest path $P_2$, interior vertices of which have all degree 2.
Analogously start at $u_3$, proceed with the above mentioned neighbour of $u_3$ and construct a longest path $P_3$, interior vertices of which have all degree 2.
Finally, let $i$ be the smallest index such that $a_ia_{\ell-1-i}\in E'$.
Then $P_2$ has length $i$ while $P_3$ has length $i+1$.
Hence, every automorphism of $H_{n,E'}$ must fix also $u_2$ and $u_3$.
Consequently, every automorphism of $H_{n,E'}$ fixes all the vertices of $H_{n,E'}$, and so $H_{n,E'}$ and $H_{n,E^*}$ are not isomorphic graphs.

Since $E_n$ has $\frac{\ell}{2}-2=\lfloor\frac{n-7}2\rfloor-2=\lfloor\frac{n-11}2\rfloor$ edges and every subset gives different graph, there are exactly $2^{\lfloor\frac{n-11}2\rfloor}$ nonisomorphic graphs $H_{n,E'}$.
\end{proof}

We conclude the paper by presenting another, more general, construction of graphs $G$ with $\Delta(G)=4$, $\chi(G)=3$, $\ivs(G)=3$ and $\vs(G)=2$. Let $H$ be a bipartite graph with $\Delta(H)\le4$ such that there exists a cycle $C_{2k}\subseteq H$ with $k\ge3$
and a pair $a,b\in V(C_{2k})$ of non-adjacent vertices in $H$ on odd distance $d_H(a,b)$ and having $\deg_H(a)=\deg_H(b)=2$. Take the union of $H$ with a disjoint triangle $K_3=uvw$ and add the edges $av,bv,aw,bw$. Denote the resulting graph by $G$ (see Figure~\ref{fig:k}).

\begin{figure}[htp!]
	$$
    \includegraphics[scale=0.5]{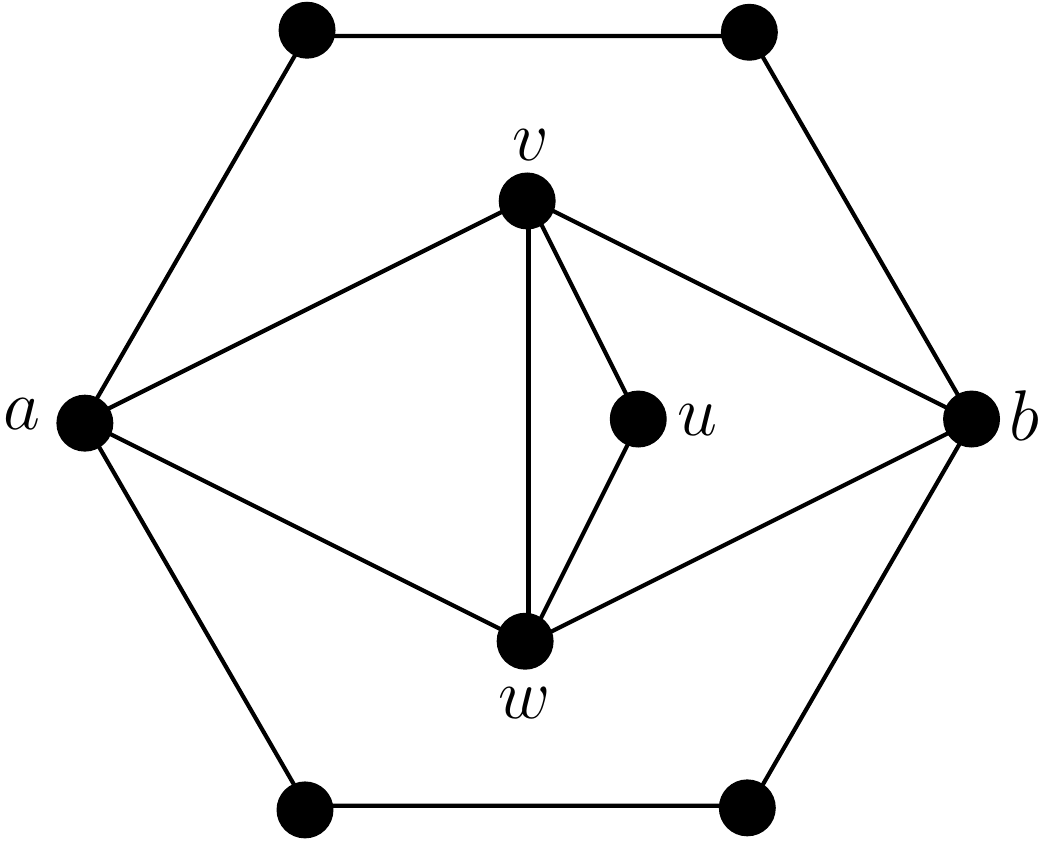}
	$$
	\caption{The graph $G$ if $H=C_6$.}
	\label{fig:k}
\end{figure}

\begin{proposition}
    \label{bipartite}
If $H$ is of order $m$ then $G\in S_{m+3}$.
Moreover, if $H$ is 2-connected (resp. planar) then $G$ has the same feature.
\end{proposition}

\begin{proof}
Clearly, the order of $G$ is $m+3$.
 As $\Delta(H)\le4$ and $\deg_H(a)=\deg_H(b)=2$ we have $\Delta(G)=4$. In view of the triangle $uvw$, the graph $G$ is not bipartite. In order to show $\chi(G)=3$, note that $G-\{u,v,w\}$ is bipartite. Take a proper $2$-colouring $\varphi$ of $G-\{u,v,w\}$ with colours $1$ and $3$, say $\varphi(a)=1$ and $\varphi(b)=3$ (here we use that $d(a,b)$ is odd). Now change the colour of $b$ to $1$ and the colour of every $c\in N_{G-\{u,v,w\}}(b)$ to $2$. Note that by assigning the colour $1$ to $u$, the colour $2$ to $v$ and the colour $3$ to $w$ we obtain a proper $3$-colouring of $G$.

Let us show next that $\ivs(G)=3$. Since $G-\{a,b,u\}$ is bipartite, we have $\ivs(G)\le3$.
Suppose there are non-adjacent vertices $x,y$ such that $G-\{x,y\}$ is bipartite.
In view of the triangle $uvw$, the intersection $\{x,y\}\cap\{u,v,w\}$ is a singleton.
We argue that this intersection is not the vertex $u$ due to the triangles $avw$ and $bvw$.
Let $P$ and $Q$ be the two $a-b$ paths in $C_{2k}$, and recall that both these paths are of odd lengths. Consequently, each of the cycles $C'=P\cup avb$, $C''=P\cup awb$, $C'''=Q\cup avb$, and $C''''=Q\cup awb$ is odd. Hence $\{x,y\}\cap\{v,w\}\neq\emptyset$, which further implies that $\{x,y\}\cap\{a,b\}=\emptyset$. However, then at least one of the cycles $C',C'',C''',C''''$ appears in $G-\{x,y\}$. The obtained contradiction shows $\ivs(G)=3$.

Finally, we prove that $\vs(G)=2$.
Clearly $\vs(G)\le2$, because $G-\{v,w\}$ is bipartite.
And since $\vs(G)=1$ implies $\ivs(G)=1$, we have $\vs(G)=2$.
\end{proof}

Note in passing that the order of the bound $|S_n|\ge2^{\lfloor\frac{n-11}2\rfloor}$ obtained in Theorem~\ref{thm:many} is far from (asymptotically) optimal. Propositions~\ref{pro:subdivide} and~\ref{bipartite} enable one to construct connected planar graphs within $S_n$ with considerable ease. However, establishing a more precise asymptotic estimate of $|S_n|$ was not the focus of this short article; instead, the aim was simply to point out to the existence of exponentially many graphs in $S_n$.

%

\bigskip

\bigskip\noindent\textbf{Acknowledgments.}~~The first author acknowledges
partial support by Slovak research grants VEGA 1/0567/22, VEGA 1/0206/20,
APVV--19--0308, APVV--17--0428. All authors acknowledge partial support of the Slovenian research agency
ARRS program P1-0383 and ARRS project J1-3002.

\end{document}